\documentclass[12pt,reqno,a4paper]{amsart}
\usepackage{
    amsmath,  amsfonts, amssymb,  amsthm,   amscd,
    gensymb,  graphicx, comment,  etoolbox, url,
    booktabs, stackrel, mathtools,enumitem, mathdots,  microtype, lmodern,    mathrsfs, graphicx, tikz,  longtable,tabularx, float, tikz, pst-node, tikz-cd, multirow, tabularx, amscd,  bm, array, makecell, diagbox, booktabs,ragged2e, caption, subcaption }
\usepackage{makecell,slashbox}

\usepackage{xcolor}
\usepackage[utf8]{inputenc}
\usepackage{microtype, fullpage, wrapfig,textcomp,mathrsfs,csquotes,fbb}
\usepackage[colorlinks=true, linkcolor=blue, citecolor=blue, urlcolor=blue, breaklinks=true]{hyperref}
\usepackage[capitalise]{cleveref}
\usepackage{todonotes}
\usetikzlibrary{positioning}
\usetikzlibrary{shapes,arrows.meta,calc}

\usetikzlibrary{arrows}

\newtheorem{theorem}{Theorem}[section]

\newtheorem{corollary}[theorem] {Corollary}
\newtheorem{definition}[theorem]{Definition}
\newtheorem{example}[theorem]{Example}
\newtheorem{lemma}[theorem]{Lemma}

\newtheorem{proposition}[theorem]{Proposition}
\newtheorem{remark}[theorem]{Remark}

\newcommand\R{\mathbb{R}}
\newcommand\Z{\mathbb{Z}}
\newcommand\C{\mathbb{C} }

\newcommand{\TC}{\mathrm{TC}}

\newcommand{\sct}{\mathrm{secat}}
\newcommand{\ct}{\mathrm{cat}}

\newcommand{\wgt}{\mathrm{wgt}}

\newcommand{\gen}{\mathfrak{genus}}

\newcommand{\RomanNumeralCaps}[1]
    {\MakeUppercase{\romannumeral #1}}

\newcolumntype{x}[1]{>{\centering\arraybackslash}p{#1}}

\begin{document}

\begin{center}
{{\Large \textbf { \sc  Group actions and higher topological complexity of lens spaces }}
\\

\medskip

{\sc Navnath Daundkar}\\
{\footnotesize Indian Institute of Science Education and Research Pune, India}\\

{\footnotesize e-mail: {\it navnath.daundkar@acads.iiserpune.ac.in}}}\\

\end{center}

\medskip

\begin{center}
  {\sc Abstract}\\
\end{center}
 In this paper, we obtain an upper bound on the higher topological complexity of the total spaces of fibrations. As an application, we improve the usual dimensional upper bound on higher topological complexity of total spaces of some sphere bundles. 
We show that this upper bound on the higher topological complexity of the total spaces of fibrations can be improved using the notion of higher subspace topological complexity. 
We also show that the  usual dimensional upper bound on the higher topological complexity of any path-connected space can be improved in the presence of positive dimensional compact Lie group action. 
We use these results to compute the exact value of higher topological complexity of lens spaces in many cases.

\hrulefill



\section{Introduction}   \label{sec:intro}
Let $X$ be the configuration space of a mechanical system. The points of $X$ represent the states of the system. Then a \emph{motion planning algorithm} for a mechanical system is a function that associates a pair of states $(x,y)$, a continuous motion from $x$ to $y$. In other words, a motion planning algorithm is a section of the free path space fibration
\[\pi:X^I\to X\times X ~~ \text{ defined by }~~\pi(\gamma)=(\gamma(0),\gamma(1)), \]
where $X^I$ is a free path space of $X$ with a compact open topology.
The \emph{topological complexity} of a space $X$, denoted $\TC(X)$ is the least integer $r$ for which $X\times X$ is covered by open sets $\{U_1,\dots, U_r\}$,  such that each $U_i$ admits a continuous local section of $\pi$. The notion of topological complexity was introduced by Farber in \cite{FarberTC} and shown that $\TC(X)$ is a numerical homotopy invariant of a space $X$. 
It is evident that the integer $\TC(X)$ measures the complexity of the problem of finding a motion planning algorithm for a mechanical system whose configuration space is $X$ up to homotopy equivalence. This numerical homotopy invariant was studied widely for the last two decades.

Rudyak \cite{RUD2010} subsequently introduced the higher analogue of the topological complexity, known as \emph{higher (or sequential) topological complexity}. 
For a path connected space $X$, consider the fibration 
\begin{equation}\label{eq: pik fibration}
\pi_k: X^I\to X^k  \text{ defined by }  \pi_k(\gamma)=\bigg(\gamma(0), \gamma(\frac{1}{k-1}),\dots,\gamma(\frac{k-2}{k-1}),\gamma(1)\bigg).   \end{equation}
The \emph{higher topological complexity} is the least integer $r$ for which $X^k$ is covered by open sets $\{U_1,\dots, U_r\}$,  such that each $U_i$ admits a continuous local section of $\pi_k$.
Note that if $k=2$, $\TC_2(X)$ coincides with the integer $\TC(X)$. 
If $X$ is the configuration space of a mechanical system, then the integer $\TC_k(X)$ can be thought of as a minimum number of continuous rules that are needed to program the given system so that while moving from any initial state to final state, it visits $k-2$ states on the way.

The topological complexity is closely related to the \emph{Lusternik-Schnirelmann category} (\emph{LS-category}) of a space $X$, denoted as $\mathrm{cat}(X)$. LS-category of a space $X$, is the smallest integer $r$ such that $X$ can be covered by $r$ open subsets $V_1, \dots, V_r$ with each inclusion $V_i\xhookrightarrow{} X$ is null-homotopic.
In particular, the following inequalities were proved in  \cite{gonzalezhighertc}, 
\[\mathrm{cat}(X^{k-1})\leq \TC_k(X)\leq \mathrm{cat}(X^k).\]

All these invariants are particular cases of a more general invariant known as sectional category (see \cite{secat}).
The \emph{sectional category} of a map $p:E\to B$, denoted $\sct(p)$ is the least integer $r$ for which $B$ can be covered by $r$ many open sets $W_1,\dots,W_r$ with continuous maps $s_i:W_i\to E$ such that $p\circ s_i$ is homotopic to the inclusion $\iota: W_i\to B$ for each $1\leq i\leq r$. 
If $p: E\to B$ is a fibration, then $\sct(p)$ coincides with another invariant called \emph{Schwarz genus}, of a fibration, denoted $\gen(p)$ (see \cite{Svarc61}). 
For example, $\TC_k(X)=\gen(\pi_k)$ and $\ct(X)=\sct(\iota: \ast\xhookrightarrow{} X)$.

In general, determining the exact value of these invariants is a difficult task. Since the last two decades, many mathematicians have contributed significantly towards approximating these invariants with bounds. 
More precisely, cohomological lower bound on the topological complexity was given by Farber in \cite[Theorem 7]{FarberTC} and this was generalized by Rudyak for the higher topological complexity in \cite[Proposition 3.4]{RUD2010}. 
More tight lower bounds on Schwarz genus and consequently on (higher) topological complexity, can be obtained using the concept of weights of cohomology classes with respect to a fibration (see \cite{GrantCweights}, \cite{daundkar2023htc}, \cite{cwtFadellHusseini}). 
For a paracompact space, there is a usual \emph{dimensional upper bound} on the higher topological complexity given as follows:
\begin{equation}\label{eq: usual dim ub}
 \TC_k(X)\leq k\cdot\mathrm{dim}(X)+1.   
\end{equation}
The more general upper bound in terms of homotopy-dimension and the connectivity of a space is given in  \cite[Theorem 3.9]{gonzalezhighertc} as 
\begin{equation}\label{eq: dim conn ub}
    \TC_k(X)\leq  \frac{k\cdot\mathrm{hdim}(X)}{\mathrm{conn}(X)+1}+1.
\end{equation}

For a fibration $p:E\to B$ with fibre $F$, Farber and Grant \cite{GrantCweights} gave the following upper bound on $\TC(E)$
\begin{equation}\label{eq: tc ub on E}
    \TC(E)\leq \TC(F)\cdot \ct(B\times B).
\end{equation}
Then using the upper bound given in \eqref{eq: tc ub on E}, they improved the dimensional upper bound on the topological complexity of the lens spaces (see \cite[Corollary 10]{GrantCweights}) by realizing lens spaces as the total spaces of $S^1$-fibrations over complex projective spaces.
We prove the higher analogue of \eqref{eq: tc ub on E} (see \Cref{thm: tck total space}) and show that this cannot be utilized to improve the  dimensional upper bound on the higher topological complexity of lens spaces (see \Cref{rmk: lens}). 
Then we observe that the higher analogue of \eqref{eq: tc ub on E} can be further improved using the notion of higher subspace topological complexity (see \Cref{thm: tck ub total space sub tc}).
We also study some basic properties of this newly introduced notion in \Cref{sec2}.

Grant \cite[Corollary 5.3]{grantfibsymm} improved the dimensional upper bound on topological complexity in the presence of the positive dimensional compact Lie group actions on spaces.
In this paper, we prove the higher analogue of this result (see \Cref{cor: Gact ub}). 

The problem of computing topological complexity of lens spaces has received a significant attention in the past. 
Gonz\'{a}lez \cite{Gonzalezlens1} showed that in some cases the topological complexity of lens spaces can be related to their immersion dimension. Later Farber and Grant \cite{GrantCweights} computed the exact values of topological complexity of high torsion lens spaces. 
The higher topological complexity of lens spsces, in contrast to the higher topological complexity of projective sapces, has not received much attention in the literature.
In this paper, we have made some progress towards solving this question. In particular, we use the results proved in the earlier part of this paper to compute the 
exact value of the higher topological complexity of lens spaces in many cases (see \Cref{thm: exactvalue}). 

\emph{The results appeared in this paper are influenced by the corresponding results for the topological complexity obtained by Farber-Grant in \cite{GrantCweights} and in \cite{grantfibsymm} by Grant.} 

\section{Higher subspace topological complexity}\label{sec2}
The basic properties of subspace (or relative) topological complexity  have been studied  by Farber and Grant in \cite{Farberbook}, \cite{Farberinstab}, \cite{grantfibsymm} etc. In this section, we introduce the notion of higher subspace topological complexity\footnote{For $k=2$, the notion coincides with Farber's relative topological complexity (see \cite[Section 4.3]{FarberTC}). }
 and study some of its properties.

Recall the definition of the fibration $\pi_k: X^I\to X^k$ from \eqref{eq: pik fibration}.
\begin{definition}
Let $A\subseteq X^k$ and $\pi|_{k,A}:\pi_k^{-1}(A)\to A$ denotes the restriction to paths $\gamma$ such that a tuple $\big(\gamma(0),\gamma(\frac{1}{k-1}),\dots,\gamma(\frac{k-2}{k-1}),\gamma(1)\big)\in A$. 
Then the higher subspace topological complexity of $A$ in $X$ is defined as $\TC_{k,X}(A):=\sct(\pi|_{k,A})$.
\end{definition}

\begin{definition}
Let $X$ be a path connected space with $A\subseteq X$. 
The subspace category of $A$ in $X$, denoted $\ct_X(A)$, is the smallest positive integer $r$ for which $A$ admits an open cover $\{V_1,\dots, V_r\}$ such that the compositions $V_i\xhookrightarrow{} A \xhookrightarrow{} X$ are null-homotopic for each $1\leq i\leq k$. 
\end{definition}

One can note that, $\TC_{k,X}(X^k)=\TC_k(X)$. 
Now we mention some obvious inequalities consisting  $\TC_{k,X}(A)$, $\TC_k(X)$ and $\ct_{X^k}(A)$. Recall that such inequalities for $k=2$ already appeared in Farber's book \cite{Farberbook}. 
Let $A\subseteq X^k$. Then we have
\begin{enumerate}
    \item $\TC_{k,X}(A)\leq \TC_k(X)$.
    \item $\TC_{k,X}(A)\leq \ct_{X^k}(A)$.
    \item If $A\subseteq B\subseteq X^k$, then $\TC_{k,X}(A)\leq \TC_{k,X}(B)$.
    \item Let $Y\subseteq X$. Then $\TC_{k,X}(Y^k)\leq \TC_k(Y)$.
\end{enumerate}

We now prove the higher analogue of \cite[Lemma 4.21]{Farberbook}. 
\begin{lemma}\label{lem: basic tfae}
Let $A\subseteq X^k$. Then  the following are equivalent:
\begin{enumerate}
    \item $\TC_{k,X}(A)=1$
    \item The projections from $A$ to the $i^{\text{th}}$ factor of $X^k$ are homotopic for $1\leq i\leq k$.
    \item The inclusion $A\xhookrightarrow{} X^k$ is homotopic to a map with values in the diagonal $\Delta(X)\subseteq X^k$.
    \end{enumerate}
\end{lemma}
\begin{proof}($1\implies 2$). Let $A\subseteq X^k$ such that $\TC_{k,X}(A)=1$. Then we have a section \[s: A\to \pi^{-1}(A)\subseteq X^I\] of $\pi|_{k,A}$.  Let $pr_i:A\xhookrightarrow{} X^k\to X$ be a projection onto the $i^{\text{th}}$ factor of $X^k$.  It is enough to show that $pr_i\simeq pr_{i+1}$ for $1\leq i\leq k-1$. Define 
\[H_i(\bar{a},t):=s(\bar{a})\bigg( \frac{t+i-1}{k-1}\bigg),\] where $\bar{a}=(a_1,\dots,a_k)\in A$. 
Then $H_i(\bar{a},0)=s(\bar{a})(\frac{i-1}{k-1})=pr_i(\bar{a})$ and
$H_i(\bar{a},1)=s(\bar{a})(\frac{i}{k-1})=pr_{i+1}(\bar{a})$. This shows that $pr_i\simeq pr_{i+1}$. Consequently, all projections from $A$ to the $i^{\text{th}}$ factor of $X^k$ are homotopic for $1\leq i\leq k$.\\

($2\implies 3$) 
Since any two projections from $A$ to the $i^{\text{th}}$ factor of $X^k$ are homotopic, we can choose a homotopy $H_i:A\times I\to X$ be homotopies from $pr_i$ to $pr_{i+1}$ for $1\leq i\leq k-1$.
Note that, for $\bar{a}\in A$, we have $H_i(\bar{a},t)$ is a path from $a_i$ to $a_{i+1}$.
For $1\leq j\leq n-1$, we define a concatenation of paths as  $\beta^j_{\bar{a}}=\ast_{i=1}^{j}\alpha^i_{\bar{a}}$, where $\alpha^i_{\bar{a}}(t)=H_i(\bar{a},t)$. Note that $\beta^j_{\bar{a}}$ is a path from $a_1$ to $a_{j+1}$.
Now define a homotopy $H:A\times I\to X^k$ as
\[H(\bar{a},t)=(H_1(\bar{a},t(1-t)),\beta^1_{\bar{a}}(1-t),\dots,\beta^{k-1}_{\bar{a}}(1-t)).\]
Then observe that \[H(\bar{a},0)=(H_1(\bar{a},0),\beta^1_{\bar{a}}(1),\dots,\beta^{k-1}_{\bar{a}}(1))=(a_1,\dots,a_k)=\bar{a}\] and 
    \[H(\bar{a},1)=(H_1(\bar{a},0),\beta^1_{\bar{a}}(0),\dots,\beta^{k-1}_{\bar{a}}(0))=(a_1,\dots,a_1)\in \Delta(X).\]\\

($3\implies 1$) Let $\iota_A:A\xhookrightarrow{} X^k$ be the inclusion map and $f: A\to \Delta(X)\subseteq X^k$ be a map such that $\iota_A\simeq f$. That is there is a homotopy $H:A\times I\to X^k$ such that $H(\bar{a},0)=\bar{a}$ and $H(\bar{a},1)=(c_{\bar{a}},\dots,c_{\bar{a}})$. Observe that given $\bar{a}$, $H(\bar{a},t)=(\gamma_1(t),\dots, \gamma_{k}(t))$ such that $\gamma_i(0)=a_i$ and $\gamma_i(1)=c_{\bar{a}}$. 
We need to define $s: A\to \pi^{-1}(A)$ such that $s(\bar{a})(\frac{i}{k-1})=a_{i+1}$ for $0\leq i\leq k-1$.
For $\bar{a}\in A$, define \[s(\bar{a})(t)=\begin{cases}  
\gamma_1*\bar{\gamma}_2((k-1)t) & t\in [0,\frac{1}{k-1}]\\
\gamma_2*\bar{\gamma}_3((k-1)t-1) & t\in [\frac{1}{k-1},\frac{2}{k-1}]\\ 
\hspace{1cm}.&.\\
\hspace{1cm}.&.\\
\hspace{1cm}.&.\\
\gamma_r*\bar{\gamma}_{r+1}((k-1)t-r+1) & t\in [\frac{r-1}{k-1},\frac{r}{k-1}]\\
\hspace{1cm}.&.\\
\hspace{1cm}.&.\\
\hspace{1cm}.&.\\
\gamma_{k-1}*\bar{\gamma}_{k}((k-1)t-(k-2))& t\in [\frac{k-2}{k-1},1],
\end{cases} \] where $\bar{\gamma_j}$ be the reverse path of $\gamma_j$ for $1\leq r\leq k$. Now one can see that 
\[s(\bar{a})(0)=\gamma_1*\bar{\gamma_2}(0)=a_1~~ \text{ and }  ~~s(\bar{a})(\frac{i}{k-1})=\gamma_i*\bar{\gamma}_{i+1}(1)=a_{i+1} ~~ \text{for}  1\leq i\leq k-1.\]
This shows that $s: A\to \pi^{-1}(A)$ is a section of $\pi|_{k,A}$. Consequently, $\TC_{k,X}(A)=1$. 
\end{proof}

The following result is a higher analogue of \cite[Lemma 2.6]{grantfibsymm} which relates the subspace category and higher subspace topological complexity. 
\begin{lemma}\label{lem: catx tckx cat}
Let $X$ be a path connected space and $A\subseteq X$. Then 
\begin{equation}\label{eq: tckx cat}
\ct_{X^{k-1}}(A^{k-1})\leq \TC_{k,X}(A^k)\leq \ct_{X^k}(A^k).
\end{equation}
\end{lemma}
\begin{proof}
The right inequality of \eqref{eq: tckx cat} is straightforward. 
Note that if the inclusion $U\xhookrightarrow{}A^k$ is nullhomotopic in $X^k$, then it can be homotoped to a map which takes values in the diagonal $\Delta(X)\subseteq X^k$. Therefore, one can define a section  $s_U: U\to \pi^{-1}(A^k)$ of $\pi|_{k,A^k}:\pi^{-1}(A^k)\to A^k$ using the same strategy as in the  proof of part ($3\implies 1$) of \Cref{lem: basic tfae}.

Let $U\subseteq A^k$ with a section $s: U\to \pi^{-1}(A)$ of $\pi|_{k,A^k}$. 
We fix an element $a_0\in A$ and  define \[V=\{\bar{x}\in X^{k-1} \mid (a_0,\bar{x})\in U\},\] where $\bar{x}=(x_2,\dots,x_k)$. 
Then note that $V$ is an open subset of $A^{k-1}$. Now we show that $V$ is contractible in $X^{k-1}$. Here we don't distinguish between $\{a_0\}\times X^{k-1}$ and $X^{k-1}$. Let $\bar{x}_{a_0}=(a_0,x_2,\dots,x_k)$ and $\beta^j_{\bar{x}_{a_0}}$ be a path from $a_0$ to $x_j$ (see the proof of \Cref{lem: basic tfae} for the construction of such path). Then we define homotopy $H: V\times I\to X^{k-1}$ as 
\[H(\bar{x},t)=\bigg(s(\bar{x}_{a_0})(\frac{t(1-t)}{k-1}),s(\bar{x}_{a_0})(\frac{1-t)}{k-1}), \beta^2_{\bar{x}_{a_0}}(1-t),\dots, \beta^{k-1}_{\bar{x}_{a_0}}(1-t) \bigg).\]
Then \[H(\bar{x},0)=\bigg(s(\bar{x}_{a_0})(0),s(\bar{x}_{a_0})(\frac{1}{k-1}), \beta^2_{\bar{x}_{a_0}}(1),\dots, \beta^{k-1}_{\bar{x}_{a_0}}(1) \bigg)=(a_0,x_2,\dots,x_k),\] and 
\[H(\bar{x},1)=\bigg(s(\bar{x}_{a_0})(0),s(\bar{x}_{a_0})(0), \beta^2_{\bar{x}_{a_0}}(0),\dots, \beta^{k-1}_{\bar{x}_{a_0}}(0) \bigg)=(a_0,a_0,\dots,a_0).\] 
This shows that the inclusion $V\xhookrightarrow{} A^{k-1}$ is nullhomotopic in $X^{k-1}$. Thus conclude the proof of the left inequality of \eqref{eq: tckx cat}.
\end{proof}

A metrizable space $X$ is called an \emph{Euclidean neighborhood retract} (ENR) if there is an open set $V\subseteq \R^n$ for some $n$, and a map $h: X\to \R^n$ which is a homeomorphism onto its image $h(X)$ such that $h(X)$ is a retract of $V$.
The following result is analogous to the \cite[Lemma 2.7]{grantfibsymm}. Although the proof is almost similar, include here for the sake of completeness.
\begin{lemma}\label{lemma: local sec}
Let $X$ be a normal ENR and $A$ be a closed subset of $X^k$ such that $\TC_{k,X}(A)\leq n$. Then there exist open sets $W_1,\dots, W_n$ in $X^k$ such that $A\subseteq\cup_{i=1}^nW_i$ and $\pi:X^I\to X^k$ admits a local section on each $W_i$.    
\end{lemma}
\begin{proof}
 Since $\TC_{k,X}(A)\leq n$, there exist an open cover  $U_1,\dots,U_n$ of $A$ such that each of $U_i$ admits a continuous section of $\pi$. 
 Note that $A$ is normal, therefore we can have another open cover $V_1,\dots,V_n$ of $A$ such that $V_i\subseteq \overline{V_i}\subseteq U_i$ for all $i=1,\dots,n$.
 It follows from \Cref{lem: basic tfae}, that the projections $pr_i:\overline{V_i}\to X$ are homotopic for $i=1,\dots,k$. 
 
 Since $X^k$ is normal ENR, there exist another open cover $W_1,\dots,W_n$ of $A$ such that $\overline{V_i}\subseteq W_i$. It follows from \cite[Exercise IV 8.2]{Doldbook}  that $pr_i:W_i\to X$ are homotopic. This proves the result.
\end{proof}

 \section{Upper bounds on the higher topological complexity of total spaces of fibrations}
For a Hurewicz fibration $F\xhookrightarrow{} E\to B$, Farber and 
Grant in \cite{GrantCweights} showed that \[\TC(E)\leq \TC(F)\cdot \ct(B\times B).\] 
Grant later \cite{grantfibsymm} showed that this bound can be improved using the notion of subspace topological complexity. 
In this section, we obtain a similar upper bound in the case of higher topological complexity and improve it using higher subspace topological complexity.
We also remark that, this upper bound may not improve the usual dimensional upper bound on higher topological complexity of lens spaces.

\begin{theorem}\label{thm: tck total space}
Let $p: E\to B$ be a fibration with fibre $F$. Then 
\begin{equation}\label{eq: ub total space}
    \TC_k(E)\leq \TC_k(F)\cdot\ct(B^k).
\end{equation}    
\end{theorem}
\begin{proof}
Let $\{U_1,\dots,U_m\}$ be a categorical cover for $B^k$, an open cover $\{V_1,\dots,V_{r'}\}$ of $F^k$ with maps $s_i:V_i\to F^I$ as sections of $\pi_k:F^I\to F^k$ and a product fibration $p^k:E^k\to B^k$. 
In particular, we have the following diagram, where horizontal arrows denotes the usual inclusion maps and without loss of generality both vertical free path space fibrations are denoted by $\pi_k$. 
\[ \begin{tikzcd}
F^I \arrow{r}{} \arrow[swap]{d}{\pi_k} & E^I \arrow{d}{\pi_k} \\%
F^k \arrow{r}{}&E^k\arrow{d}{p^k}\\%
& B^k
\end{tikzcd}
\]
Since each $U_j$ is categorical, we have $H_j: U_j\times I\to B^k$ such that $H_j(\bar{x},0)=\bar{x}$ and $H_j(\bar{x},1)=(\bar{x}_0)$, where $\bar{x}=(x_1,\dots,x_k)$ and $\bar{x}_0=(x_0,\dots,x_0)$. 
For a fixed $\bar{x}$ we have $H_j(\bar{x},t)$ is a path in $B^k$. 
Therefore, for $1\leq j\leq m$, we have $H_j(\bar{x},t)=(\alpha^1_{j\bar{x}},\dots,\alpha^k_{j\bar{x}})$, where $\alpha^i_{j\bar{x}}$ is a path in $B$ such that $\alpha^i_{j\bar{x}}(0)=x_i$ and $\alpha^i_{j\bar{x}}(0)=x_0$ for $1\leq i\leq k$ and $1\leq j\leq m$.

Now we consider the following space
as mentioned in \cite[Chapter 2, Section 7]{spanier}, 
\[\bar{B}=\{(e,\gamma)\in E\times B^I : \gamma(0)=p(e)\}\] and a lifting function $\lambda: \bar{B}\to E^I$ defined by \[\lambda(e,\gamma)=\tilde{\gamma}_e,\] where $\tilde{\gamma}_e$ is a lift of $\gamma$ such that $p\circ \tilde{\gamma}_e=\gamma$. 

Let $\bar{e}=(e_1,\dots,e_k)\in E^k$ such that $p^k(\bar{e})=\bar{x}$. Then for each $1\leq i\leq k$, the end points of $\lambda(e_i,\alpha^i_{j\bar{x}})$ that is $\lambda(e_i,\alpha^i_{j\bar{x}})(1)$ are in $p^{-1}(x_0)=F$. Let $\bar{e}\in E^k$. 
Now for $1\leq j\leq m$, we define a map $\sigma_j: (p^k)^{-1}(U_j)\to F^k$ by 
\[\sigma_j(\bar{e})=(\lambda(e_1,\alpha^1_{j\bar{x}})(1),\dots,\lambda(e_k,\alpha^k_{j\bar{x}})(1)).\] 

Consider the open sets $W_{ij}=\sigma_j^{-1}(V_i)$ for $1\leq i\leq r'$ and $1\leq j\leq m$. One can see that these sets form an open cover of $E^k$. 
\[ \begin{tikzcd}
F^I \arrow{r}{} \arrow[swap]{d}{\pi_k} & E^I \arrow{d}{\pi_k} \\%
F^k  & (p^k)^{-1}(U_j)\supseteq W_{ij}\arrow{l}{\sigma_j} \arrow{d}{p^k}\\%
& U_j
\end{tikzcd}
\]
Now we show that, each $W_{ij}$ admits a section $\tau_{ij}$ of $\pi_k$. 
Let $\bar{e}=(e_1,\dots,e_k)\in W_{ij}$. For $1\leq l\leq k$, denote $\beta^j_l=\lambda(e_l,\alpha^l_{j\bar{x}})$ as a path which starts from $e_l$ and whose end point is in $F$.
A path $\bar{\beta^j_{i}}$ denotes the inverse path of $\beta^j_i$ for $1\leq i\leq k$ and $1\leq j\leq m$.
For $1\leq i\leq r'$ and $1\leq j\leq m$, we defined $\tau_{ij}:W_{ij}\to E^I$ as  
\begin{equation}\label{eq: gj}
 \tau_{ij}(\bar{e})(t):=\begin{cases}  
\beta^j_1*s_i*\bar{\beta^j_2}((k-1)t) & t\in [0,\frac{1}{k-1}]\\
\beta^j_2*s_i*\bar{\beta^j_3}((k-1)t-1) & t\in [\frac{1}{k-1},\frac{2}{k-1}]\\ 
\hspace{1cm}.&.\\
\hspace{1cm}.&.\\
\hspace{1cm}.&.\\
\beta^j_r*s_i*\beta^j_{r+1}((k-1)t-r+1) & t\in [\frac{r-1}{k-1},\frac{r}{k-1}]\\
\hspace{1cm}.&.\\
\hspace{1cm}.&.\\
\hspace{1cm}.&.\\
\beta^j_{k-1}*s_i*\bar{\beta^j_{k}}((k-1)t-(k-2))& t\in [\frac{k-2}{k-1},1],
\end{cases}   
\end{equation}
where $*$ denotes the concatenation of paths.
Observe that $\tau_{ij}$ is well defined and $\tau_{ij}(\frac{l}{k-1})=e_{l+1}$ for $1\leq l\leq k-1$. This gives $\pi_E\circ \tau_{ij}=\iota_{W_{ij}}$, where $\iota_{W_{ij}}$ is the inclusion of $W_{ij}$ into $E^k$.
\end{proof}

The following theorem improves the upper bound on the higher topological complexity of the total space of the fibration that was obtained in \Cref{thm: tck total space}.

\begin{theorem}\label{thm: tck ub total space sub tc}
Let $X\to E^k\stackrel{p}\longrightarrow Y$ be a fibration such that $X$, $E$ and $Y$ are path connected. Then 
\begin{equation}
\TC_{k}(E)\leq \ct(Y)\cdot \TC_{k,E}(X).  \end{equation}
In particular, for a fibration $F^k\to E^k\to B^k$
\begin{equation}\label{eq: tck ub hsub tc}
\TC_{k}(E)\leq \ct(B^k)\cdot \TC_{k,E}(F^k).    
\end{equation}
\end{theorem}
\begin{proof}
Suppose $\ct(Y)=m$. Then there is a categorical cover $\{V_1,\dots,V_m\}$ of $Y$. Without loss of generality assume that each $V_i$ contracts to a point $y_0\in Y$. 
Now consider the open cover $\{A_j=p^{-1}(V_j) \mid 1\leq j\leq m\}$ of $E^k$. Since each $V_j$'s are contractible in $Y$, we have $p^{-1}(V_j)\simeq V_j\times X$. Therefore, for each $1\leq j\leq m$, there exist homotopy  $H^j:A_j\times I\to E^k$ such that $H^j(\bar{x},0)=\bar{x}$ and $H^j(\bar{x},1)=p^{-1}(y_0)=X$. 
Suppose $\TC_{k,E}(X)=n$. Then consider the open cover $\{W_1,\dots, W_n\}$ of $X$ such that each $W_i$ admits a section of $\pi$. 
Let $\alpha_j=H^j(-,1)$ be a map from $A_j\to X$.  We consider the open sets defined by $U_{ij}=\alpha_j^{-1}(W_i)$ for $1\leq j\leq m$ and $1\leq i\leq n$. Then note that $\{U_{ij}\mid 1\leq j\leq m , 1\leq i\leq n\}$ forms an open cover of $E^k$. Now we show that each $U_{i,j}$ admits continuous section of $\pi_k$.

Note that for a fixed $\bar{e}=(e_1,\dots,e_k)\in E^k$, $H^j(\bar{e},t)$ is a path in $E^k$ whose end point lies in $X$. Therefore, for $1\leq j\leq m$ we have  \[H^j(\bar{e},t)=(\alpha^j_{1\bar{e}}(t),\dots, \alpha^j_{k\bar{e}}(t))\] such that $\alpha^j_{i\bar{e}}(0)=e_i$ for $1\leq i\leq k$ and $(\alpha^j_{1\bar{e}}(1),\dots, \alpha^j_{k\bar{e}}(1))\in X$.
Without loss of generality we forget the dependence of $\bar{e}$ and write $\gamma^j_{i}=\alpha^j_{i\bar{e}}$ for $1\leq i\leq k$.
Define $\mathfrak{t}_{ij}:U_{ij}\to E^I$ as follows
\begin{equation}\label{eq: section on Uij}
 \mathfrak{t}_{ij}(\bar{e})(t):=\begin{cases}  
\gamma^j_1*s_i(H^j(\bar{e},1))*\bar{\gamma}^j_2((k-1)t) & t\in [0,\frac{1}{k-1}]\\
\gamma^j_2*s_i(H^j(\bar{e},1))*\bar{\gamma}^j_3((k-1)t-1) & t\in [\frac{1}{k-1},\frac{2}{k-1}]\\ 
\hspace{1cm}.&.\\
\hspace{1cm}.&.\\
\hspace{1cm}.&.\\
\gamma^j_r*s_i(H^j(\bar{e},1))*\bar{\gamma}^j_{r+1}((k-1)t-r+1) & t\in [\frac{r-1}{k-1},\frac{r}{k-1}]\\
\hspace{1cm}.&.\\
\hspace{1cm}.&.\\
\hspace{1cm}.&.\\
\gamma^j_{k-1}*s_i(H^j(\bar{e},1))*\bar{\gamma}^j_{k}((k-1)t-(k-2))& t\in [\frac{k-2}{k-1},1],
\end{cases}   
\end{equation}
where $*$ denote the concatenation of paths. One can see that $\mathfrak{t}_{ij}$ defined a section of $\pi_k$.
\end{proof}

\begin{remark}
Note that if $F\xhookrightarrow{} E$ is nullhomotopic, then $\TC_{k,E}(F)=1$. Then it follows from \Cref{thm: tck ub total space sub tc} that $\TC_k(E)\leq \ct(B^k)$. On the other hand $\TC_k(F)\geq k$, if $F$ is non-contractible. Therefore, the bound in \eqref{eq: tck ub hsub tc} improves the bound given in \eqref{eq: ub total space}.
Here it would be interesting to find an example of a fibration $F\to E\to B$ with non-contractible fibre $F$ such that $1<\TC_{k,E}(F)< \TC_{k}(F)$.
\end{remark}

\begin{example}
Let $\tilde{X}$ and $X$ be a path connected spaces with  $q:\tilde{X}\to X$ be a covering space. Then $q:\tilde{X}\to X$ is a fibration with a fibre $F$ is given by a discrete set. Note that a discrete set is contractible in a path connected space. Therefore, $\TC_{k\tilde{X}}(F)=1$. Consequently, we get \[\TC_k(\tilde{X})\leq \ct(X^k).\]
\end{example}

Now we show that when the upper bound given in \eqref{eq: ub total space} improves the usual dimensional upper bound on the higher topological complexity of the total spaces of some sphere bundles.
\begin{corollary}
Let $p: E\to S^{l_2}$ be a fibration with fibre homotopy equivalent to $S^{l_1}$. Then the upper bound on $\TC_k(E)$ given in \eqref{eq: ub total space} improves the usual dimensional upper bound if $k\leq l_1+l_2-1$ when $l_1$ is odd and $k\leq l_1+l_2-2$ when $l_1$ is even.
\end{corollary}
\begin{proof}
It follows from \Cref{thm: tck total space} that $\TC_k(E)\leq \TC_k(S^{l_1})\cdot \ct((S^{l_2})^k)$.
From \cite[Corollary 3.12]{gonzalezhighertc} that if $l_1$ is odd then \[\TC_k(S^{l_1})=\begin{cases}
    k& \text{ if } l_1 \text{ is odd },\\
    k+1& \text{ if } l_1 \text{ is even }
\end{cases}.\]
Now observe that $\ct((S^{l_2})^k)=k+1$. Therefore, the upper bound given in \eqref{eq: ub total space} improves the usual dimensional upper bound which is $k(\mathrm{dim}(E))+1$ if 
\[\TC_k(S^{l_1})\cdot \ct((S^{l_2})^k)< k(\mathrm{dim}(E))+1=k(l_1+l_2)+1.\]
In other words, if \[k(l_1+l_2)+1>\begin{cases} k(k+1)& \text{ if } l_1 \text{ is odd },\\
    (k+1)^2& \text{ if } l_1 \text{ is even }.    
\end{cases}\]
This gives us \[k<\begin{cases} l_1+l_2-1+1/k& \text{ if } l_1 \text{ is odd },\\
    l_1+l_2-2+1/k& \text{ if } l_1 \text{ is even }.    
\end{cases}\]
This concludes the result.
\end{proof}
Let us state the more general result.

\begin{corollary} \label{cor: ub tck}
Let $p: E\to B$ be a fibration with fibre homotopy equivalent to $S^{l}$. Then the upper bound on $\TC_k(E)$ given in  \eqref{eq: ub total space} improves the usual dimensional upper bound if $\ct(B^k)\leq l+\mathrm{dim}(B)$ when $l$ is odd and if $\ct(B^k)\leq l+\mathrm{dim}(B)-1$ when $l$ is even.
\end{corollary}
\begin{proof}
  From \Cref{thm: tck total space} and \cite[Corollary 3.12]{gonzalezhighertc}, we have \[\TC_k(E)\leq \begin{cases} k\cdot \ct(B^k)&\text{ if } l \text{ is odd },\\
(k+1)\cdot \ct(B^k)&\text{ if } l \text{ is even }.  \end{cases}\] 
First we consider the case when $l$ is
 odd. Therefore, \eqref{eq: ub total space} improves the usual dimensional upper hound if 
\[k\cdot\ct(B^k)<k\cdot(l+\mathrm{dim}(B))+1 .\] Therefore, we have $\ct(B^k)< l+\mathrm{dim}(B)+1/k$. This proves the first part of the corollary.
Now assume that, $l$ is even. Then \eqref{eq: ub total space} improves the usual dimensional upper hound if 
\[(k+1)\cdot\ct(B^k)<k\cdot(l+\mathrm{dim}(B))+1 .\]
This gives us \[\ct(B^k)<l+\mathrm{dim}(B)+ \frac{1-\ct(B^k)}{k}.\]
It follows from \cite[]{RUD2010} that $k\leq \TC_{k}(B)$. Then from \cite[Corollary 3.12]{gonzalezhighertc}, we get $k\leq \ct(B^k)$. Thus we get that $\frac{1-\ct(B^k)}{k}\leq \frac{1}{k}-1$. This shows, $\ct(B^k)<l+\mathrm{dim}(B)-1+1/k$. This concludes the result.
\end{proof}

\begin{remark}\label{rmk: lens}
Recall that we have a $S^1$-fibration $L^{2n+1}_m\to \C P^n$, where $L^{2n+1}_m$ is a lens space and  $\C P^n$ is the complex projective space. We also know that $\ct((\C P^n)^k)=kn+1$. Then using \Cref{cor: ub tck} observe that to improve the upper bound on $\TC_k(L^{2n+1}_m)$ follows from \eqref{eq: ub total space} by usual dimensional upper bound given by $k(2n+1)+1$, we must have $kn+1\leq 1+ 2n $. Therefore, we must have $k=2$. Thus, the upper bound on $\TC_k(L^{2n+1}_m)$ given by \eqref{eq: ub total space} does not improve the usual dimensional upper bound $k(2n+1)+1$ on $\TC_k(L^{2n+1}_m)$ if $k\geq 3$. 
\end{remark}

\section{Closed maps, group actions and an upper bound on higher topological complexity}\label{sec: gact}
The main aim of this section is to show how we can use group action to improve the usual dimensional upper bound on the higher topological complexity of a path connected space. First, we generalize \Cref{thm: tck ub total space sub tc} for closed maps. Then essentially we use this generalized result to achieve our aim.
We begin with stating a few lemma's which are crucial in generalizing \Cref{thm: tck ub total space sub tc} to closed maps.
\begin{lemma}[{\cite[Lemma A.4]{CLOT}}\label{lem:1}]
Let $B$ be an $n$-dimensional paracompact space and $\mathscr{U}=\{U_\alpha \}$ be an open cover of $B$.
Then there exist an open refinement $\mathscr{V}=\{V_{i\beta}\}$ for $1\leq i\leq n+1$ of $\mathscr{U}$ such that $V_{i\beta}\cap V_{i\beta'}=\emptyset$ for $\beta\neq \beta'$.
\end{lemma}

Let $B\to Y$ be a map. Then
a subset $V\subseteq B$ is called \emph{saturated} if it is an inverse image of an open set in $Y$.

\begin{lemma}[{\cite[Lemma 9.39]{CLOT}}\label{lem:2}]
 Let $q: B\to Y$ be a closed map such that for an open set $U\subseteq B$, $q^{-1}(y)\subseteq U$ for $y\in Y$. Then there is an saturated open set $V$ such that $q^{-1}(y)\subseteq V\subseteq U$.    
\end{lemma}

Now we are ready to prove one of the main results of this section which is a higher analogue of \cite[Theorem 4.3]{grantfibsymm}.

\begin{theorem}\label{thm: tckX leq dimY n}
Let $X$ be a normal ENR and $Y$ be a paracompact space with $q : X^k \to Y$ be a closed map such that
$\TC_{k,X}(q^{-1}(y)) \leq  n$ for each $y \in Y$. Then
\begin{equation}
\TC_{k}(X)\leq (\mathrm{dim}(Y)+1)\cdot n. 
\end{equation}
\end{theorem}
\begin{proof}
Let $O_y=q^{-1}(y)$. Suppose that $\TC_{k,X}(O_y)\leq n$. Then using \Cref{lemma: local sec} we have an open cover $\{U^y_1,\dots,U^y_n\}$ of $O_y$	such that each $U^y_i$ admits a local continuous section of $\pi|_k:X^I\to X^k$ for $1\leq i\leq n$. 
Let $U^y=\cup_{i=1}^{n}U^y_i$. 
Then by \Cref{lem:2}, we get a saturated open set $V^y$ such that $O_y\subseteq V^y\subseteq U^y$. 
Since $V^y$ is saturated there exist an open set $\tilde{V}^y$ such that $V^y=q^{-1}(\tilde{V}^y)$. 
Note that $\mathscr{U}=\{\tilde{V}^y\mid y\in Y\}$ forms an open cover of $Y$. 
Suppose $\mathrm{dim}(Y)=\ell$. Then by \Cref{lem:1} there exist an open refinement $\{\tilde{G}_{i\beta}\mid 1\leq i\leq \ell+1\}$ of cover $\mathscr{U}$ such that $\tilde{G}_{i}=\cup_{\beta}\tilde{G}_{i\beta}$, where $\tilde{G}_{i\beta}\cap \tilde{G}_{i\beta'}=\emptyset$ for $\beta\neq \beta'$ and $\tilde{G}_{i\beta}\subseteq \tilde{V}^y$. 
Now consider $G_i=q^{-1}(\tilde{G}_i)$ and $G_{i\beta}=q^{-1}(\tilde{G}_{i\beta})$. 
Then \[G_{i \beta}\subseteq q^{-1}(\tilde{G}_{i\beta})\subseteq q^{-1}(\tilde{V}^y)\subseteq U^y.\] 

For $1\leq i\leq \ell+1$ and $1\leq j\leq n$  we define $U_{ij}=\cup_{\beta} G_{i\beta j} $, where $G_{i\beta j} =G_{i\beta}\cap U^y_j$. 
Then note that $U_{ij}$'s form an open cover of $X^k$ and  each $U_{ij}$ admits a local continuous section of $\pi_k$.  
This shows that $\TC_k(X)\leq (\ell +1)\cdot n$.   
\end{proof}

Suppose that $G$ is a compact Lie group acting on a closed, connected smooth manifold $X$.
The orbit of an element $x\in X$ under this action is defined as
\[O(x) := \{ g\cdot x \mid g \in G \} \subseteq X.\] 
The stabiliser of $x \in X$ is the subgroup 
\[G_x := \{ g \in G \mid g\cdot x = x\}\] in $G$.  
\begin{definition}
\begin{enumerate}
    \item  The action is free if each stabiliser $G_x$ is the trivial subgroup. 
    \item The action is called semi-free, if each stabiliser $G_x$ is either trivial or $G$.  
\end{enumerate}
\end{definition}

The ﬁxed point set of a subgroup $H$ of $G$ is the subspace of $X$ deﬁned by 
\[F(H,X) := \{x \in X \mid h\cdot x = x \text{ for all } h \in H\}.\]
Let $X/G$ denote the orbit space. Since $X$ is compact, so is $X/G$. Since $G$ is compact and Hausdorff, $X/G$ is Hausdorff. 
Therefore, $p: X\to X/G$ is a
closed map. Note that if the action is free, then $p$ is a principal $G$-bundle.
The evaluation map $ev_x : G \to X$ at a point $x \in X$ deﬁned by \[ev_x(g) = g\cdot x.\] Note that the image of $ev_x$ is the orbit $O(x)$. The
induced map $q_x: G/G_x \to O(x)$ on cosets given by $q_x(gG_x) = g\cdot x$ is a homeomorphism onto the orbit follows from the compactness of $G$.
An orbit is called \emph{principal} if it is of maximal dimension. In fact, it follows from \cite[Theorem \RomanNumeralCaps 4.3.8]{Bredon} that $\mathrm{dim}(X/G)=\mathrm{dim}(X)-\mathrm{dim}(P)$, where $P$ is a principal orbit.

\begin{proposition}\label{prop: tck kdim ub}
Suppose that $G$ acts locally smoothly on $X^k$, and $TC_{k,X}(O(\bar{x}))\leq  n$ each $\bar{x}\in X^k$. If $P \subseteq X^k$ is a principal orbit, then
\begin{equation}\label{eq: tckX leq dimP n}
\TC_{k}(X)\leq (k\mathrm{dim}(X)-\mathrm{dim}(P)+1)\cdot n.
\end{equation}

If the action is free, then
\begin{equation}\label{eq: tck leq catXkn}
    \TC_k(X)\leq (\ct(X^k/G))\cdot n
\end{equation}
\end{proposition}
\begin{proof}
The inequality in \eqref{eq: tckX leq dimP n} follows from \Cref{thm: tckX leq dimY n} and observing that there is a orbit map $p: X^k\to X^k/G$ with $\mathrm{dim}(X^k/G)= k(\mathrm{dim}(X))-\mathrm{dim}(P)$.

\end{proof}

Now we prove the main theorem of this section.
\begin{theorem}\label{thm: ub on htc with gact}
Suppose $G$ acts locally smoothly on $X$ and diagonally on $X^k$ with principal orbit $P$. Suppose further that for any $\bar{x}=(x_1,\dots,x_k)\in X^k$  one of the
following conditions holds:
\begin{enumerate}
    \item $F(\cap_{i=1}^{k}G_{x_i},X)$ is path connected;
    \item $\bar{x}\in F(G,X)$.
\end{enumerate}
Then $\bar{x}\in X^k$ has 
$TC_{k,X}(O(\bar{x}))\leq 1$, and consequently
\begin{equation}\label{eq: tck kdim ub}
    \TC_k(X)\leq k\mathrm{dim}(X)-\mathrm{dim}(P)+1.
\end{equation}
\end{theorem}
\begin{proof}
 Let $\bar{x}\in X^k$ be a representative of the orbit $O(\bar{x})$. Note that, for $1\leq i\leq k$, we have $x_i\in F(G_{x_i},X)\subseteq F(\cap_{i=1}^k G_{x_i},X)$.
 If $F(\cap_{i=1}^{k}G_{x_i},X)$ is path connected, we can choose a path $\gamma$ in $F(\cap_{i=1}^{k}G_{x_i},X)$ such that $\gamma(j/(k-1))=x_{j+1}$ for $0\leq j\leq k-1$. 
 Consider the following diagram 
\[ \begin{tikzcd}
G \arrow{r}{ev_{\gamma}} \arrow[swap]{d}{} & X^I  \\%
G/(\cap_{i=1}^{k}G_{x_i})\arrow{ur}{q_{\gamma}} &O(\bar{x})\arrow{l}{q_{\bar{x}}^{-1}} \arrow[dashed]{u}{s}\\%
\end{tikzcd}
,\]
where $q_{\gamma}(g\cdot \cap_{i=1}^{k}x_i)=g\cdot\gamma$ and $q_{\bar{x}}$ is a homeomorphism. Then $s$ is defined as the composition of $q_{\gamma}$ and $q_{\bar{x}}^{-1}$. Note that $s(g\cdot \bar{x})=g\cdot \gamma$. Clearly, $s$ defines a section of $\pi|_{k,O(\bar{x})}$. Consequently, $TC_{k,X}(O(\bar{x}))\leq 1$. Then the inequality in \eqref{eq: tck kdim ub} follows from \Cref{prop: tck kdim ub}.

Now if $\bar{x}$ is a fixed point, then $O(\bar{x})=\{\bar{x}\}$. Clearly, $TC_{k,X}(\{\bar{x}\})= 1$. Then in \eqref{eq: tck kdim ub} follows from \Cref{prop: tck kdim ub}.
\end{proof}

\begin{corollary}\label{cor: Gact ub}
 If $G$ acts locally smoothly, non-trivially and semi-freely on $X$, then
 \begin{equation}\label{eq: ub htc sfree}
 TC_k(X) \leq  k \mathrm{dim}(X)-\mathrm{dim}(G) + 1  .  
 \end{equation}
If $G$ acts locally smoothly and freely on $X$, then
\begin{equation}\label{eq: ub htc free}
 TC_k(X)\leq \ct(X^k)/G \leq k \mathrm{dim}(X)- \mathrm{dim}(G) + 1.   
\end{equation}
\end{corollary}
\begin{proof}
Note that for a free and semi-free action of $G$, we have $\mathrm{dim}(P)=\mathrm{dim}(G)$.
Then  \eqref{eq: ub htc sfree} and \eqref{eq: ub htc free} are an immediate consequences of \Cref{thm: ub on htc with gact}.    
\end{proof}

\section{Higher topological complexity of lens spaces}
The problem of computining topological complexity of lens spaces has been considered by many mathematicians. For example, see \cite{Gonzalezlens1}, \cite{Gonzalezlens2}. In this section we compute higher topological complexity of high torsion lens spaces using results from previous sections.

Jaworowski \cite{S1actionlens} described the free  $S^1$ action on lens spaces. Consider $S^{2n+1}\subseteq \C^{n+1}$. Then recall that $L^{2n+1}_m=S^{2n+1}/\Z_m$. The $S^1$ action on $L^{2n+1}_m$ is defined as follows:

\[e^{2\pi i x}\cdot [(z_1,\dots,z_{n+1})]=[(e^{2\pi i x}z_1,\dots,e^{2\pi i x}z_{n+1})],\] where $(z_1,\dots,z_{n+1})\in S^{2n+1}$ and $[(z_1,\dots,z_{n+1})]$ denote a class in the quotient space $L^{2n+1}_m$.
It was shown \cite[Section 2]{S1actionlens} that this action is free and the corresponding quotient is the complex projective space $\C P^n$.
We use this to improve the usual dimensional upper bound on the higher topological complexity of lens spaces.
\begin{theorem}\label{thm: lens tck ub}
Let $L^{2n+1}_m$ be a lens space of dimension $2n+1$. Then
\[\TC_k(L^{2n+1}_m)\leq k\cdot(2n+1)=k\cdot\mathrm{dim}(L^{2n+1}_m).\]
\end{theorem}
\begin{proof}
 Note that $L^{2n+1}_m$ admits a free $S^1$-action. Therefore, by \Cref{cor: Gact ub}, we have \[\TC_k(L^{2n+1}_m)\leq k(2n+1)-\mathrm{dim}(S^1)+1=k(2n+1).\] This conclude the proof.   
\end{proof}

Farber and Grant \cite{GrantCweights,Farbergrantsymm} generalized the notion of category weight
which was introduced by Fadell and Husseini in \cite{cwtFadellHusseini}  to any fibration.
In particular, they defined the notion of $TC$-weight of a cohomology class in the $H^{\ast}(X\times X)$ corresponding to the free path space fibration  $\pi:X^I\to X\times X$. 
In \cite{daundkar2023htc}, authors defined the higher analogue of $\TC$-weight and called it \emph{$\TC_n$-weight} (see \cite[Definition 4.1]{daundkar2023htc}). 
We use this notion to obtain the tight lower bound on the higher topological complexity of lens space. We first describe the  mod-$m$ cohomology ring of lens spaces.

Let $y=B(x)$, where $B$ is a mod-$m$ Bockstein. Then the following description of the cohomology ring follows from \cite[Example 3E.2]{hatcher}.
\begin{equation}\label{eq: coho ring lens spaces}
 H^*(L^{2n+1}_m; \Z_m)=\displaystyle\frac{\Z_m[x,y]}{\left< y^{n+1},x^2-ay\right>},   
\end{equation}
where $|y|=2$, $|x|=1$ and $a=\begin{cases}
    0& \text{ if  $m$ is odd},\\
    m/2& \text{if $m$ is even}.
\end{cases}$

\begin{theorem}\label{thm: lbtck}
Let $k\geq 2$. Then 
\begin{equation}\label{eq: lbtck}
 \TC_k(L^{2n+1}_m)\geq \begin{cases} k\cdot (l+l'+1) &\text{ if } k \text{ is even },\\
(k-1)\cdot (l+l')+k+2n &\text{ if } k \text{ is odd },  \end{cases}   
\end{equation}
where $0\leq l,l'\leq n$ be any integers such that $m$ doesn't divide $\binom{l+l'}{l'}^{\lfloor k/2\rfloor}$.

\end{theorem}
\begin{proof}
We use the description of the cohomology ring of lens spaces from \eqref{eq: coho ring lens spaces} and \cite[Theorem 4.3]{daundkar2023htc} to obtain the inequalities of \eqref{eq: lbtck}. 
Let $y=B(x)\in H^2(L^{2n+1};\Z_m)$ and 
$\bar{y_i}_i=p_i^*(y)-p_{i-1}^*(y)$, where $p_i$ is a projection of $X^k$ onto its $i^{th}$ factor for $2\leq i\leq k$. Observe that $\bar{y}\in \ker(d_k^*)$
where $d_k^*:H^*((L^{2n+1}_m)^k;\Z_m)\to H^*(L^{2n+1};\Z_m)$ is the diagonal induced homomorphism. 
It follows from \cite[Theorem 4.3]{daundkar2023htc} that the $\TC_k$-weight $\wgt_{\pi}(\bar{y}_i)\geq 2$. Let $1\leq l,l'\leq n$ be two positive integers. Then note that for $2\leq i\leq k$, we have 
\[(\bar{y}_i)^{l+l'}=\sum_{j=0}^{l+l'}(-1)^{j}\binom{l+l'}{j} 1\otimes \dots \otimes 1\otimes y^{l+l'-j}\otimes y^j\otimes 1\otimes\dots \otimes 1,\] where $y^{l+l'-j}$ and $y^j$ are at the $(i-1)^{\text{th}}$ and $i^{\text{th}}$ position, respectively. 
One can see that $(\bar{y}_i)^{l+l'}\neq 0$ if $m$  doesn't divide $\binom{l+l'}{l'}$ as it contains a term $\binom{l+l'}{l'} 1\otimes \dots \otimes 1\otimes y^{l}\otimes y^{l'}\otimes 1\otimes\dots \otimes 1$ and $y^l, y^{l'}\neq 0$ for $l,l'\leq n$. 
Now observe that if $k=2k'$ and $m$ doesn't divide $\binom{l+l'}{l'}^{\lfloor k/2\rfloor}$, then the product \[\prod_{i=1}^{k'}(\bar{y}_{2i})^{l+l'}\neq 0\] as it contains the term $\binom{l+l'}{l'}^{k'}(y^l\otimes y^{l'})\otimes\dots \otimes(y^l\otimes y^{l'})$ which is not killed by any other term in the product.
Similarly, if $k=2k'+1$ and $m$ doesn't divide $\binom{l+l'}{l'}$, then the product \[\prod_{i=1}^{k'}(\bar{y}_{2i})^{l+l'}\neq 0\] as it contains the term $\binom{l+l'}{l'}^{k'}(y^l\otimes y^{l'})\otimes\dots \otimes(y^l\otimes y^{l'})\otimes 1$. Let $\bar{X}_i=p_i^*(x)-p_1^*(x)$. 
Note that $\bar{X}_i\in \ker(d^*_k)$ and the product $\prod_{i=2}^{k}\bar{X}_i\neq 0$ as it contains the term $1\otimes x\otimes \dots\otimes x$. 
One can see that if $k=2k'$ and for any integers  $0\leq l,l'\leq n$, if $m$ doesn't divide $\binom{l+l'}{l'}^{\lfloor k/2\rfloor}$, then the product 
\begin{equation}\label{eq: coho lb for tck lens}
  \prod_{i=2}^{k}\bar{X}_i\cdot\prod_{i=1}^{k'}(\bar{y}_{2i})^{l+l'}  
\end{equation}
is non-zero as it contains the term $\binom{l+l'}{l'}^{k'}y^l\otimes xy^{l'}\otimes\dots \otimes xy^l\otimes xy^{l'}$.
It follows from \cite[Theorem 4.3]{daundkar2023htc} that $\wgt_{\pi_k}(\bar{y}_{i})\geq 2$. Therefore, we have \[\wgt_{\pi_k}\bigg(\prod_{i=2}^{k}\bar{X}_i\cdot\prod_{i=1}^{k'}(\bar{y}_{2i})^{l+l'}\bigg)\geq 2k'\cdot (l+l')+k-1= k(l+l'+1)-1.\]
Therefore, we get the inequality of \eqref{eq: lbtck} when $k$ is even using \cite[Proposition 32, 33]{Farbergrantsymm}.

Similarly, if $k=2k'+1$ and for any integers  $0\leq l,l'\leq n$, the integer $m$ doesn't divide $\binom{l+l'}{l'}^{\lfloor k/2\rfloor}$, then the product 
\begin{equation}\label{eq: coho lb tck 1}
 \prod_{i=2}^{k}\bar{X}_i\cdot\prod_{i=1}^{k'}(\bar{y}_{2i})^{l+l'}   
\end{equation}
is non-zero as it contains the term $\binom{l+l'}{l'}^{k'}y^l\otimes xy^{l'}\otimes\dots \otimes xy^l\otimes xy^{l'}\otimes x$.
Let $\bar{z}=p_k(y)-p_1(y)$. Then $\bar{z}^n=\sum_{i=0}^{n}(-1)^i\binom{n}{i}y^i\otimes 1\otimes \dots \otimes 1 \otimes y^{n-i}$ is nonzero.
Now consider the product
\begin{equation}
     \prod_{i=2}^{k}\bar{X}_i\cdot\prod_{i=1}^{k'}(\bar{y}_{2i})^{l+l'}(\bar{z})^n \neq 0
\end{equation}
as it will contain a term $\binom{l+l'}{l'}^{k'}y^l\otimes xy^{l'}\otimes\dots \otimes xy^l\otimes xy^{l'}\otimes xy^n$. 
It follows from \cite[Theorem 4.3]{daundkar2023htc} that $\wgt_{\pi_k}(\bar{y_i})\geq 2$ and $\wgt(\bar{z})\geq 2$. Therefore, we have \[\wgt_{\pi_k}\bigg( \prod_{i=2}^{k}\bar{X}_i\cdot\prod_{i=1}^{k'}(\bar{y}_{2i})^{l+l'}(\bar{z})^n\bigg)\geq 2k'\cdot (l+l')+k+2n-1.\]
Therefore, we get the inequality of \eqref{eq: lbtck} when $k$ is odd using \cite[Proposition 32, 33]{Farbergrantsymm}.
\end{proof}

\begin{theorem}\label{thm: exactvalue}
If $m$ does not divide $\binom{2n}{n}^{\lfloor k/2\rfloor}$, then $\TC_k(L^{2n+1}_m)=k\cdot(2n+1)$.      
\end{theorem}
\begin{proof}
Follows from \Cref{thm: lbtck} and \Cref{thm: lens tck ub}.    
\end{proof}

\begin{remark}
Observe that, for $k=2$ \Cref{thm: lbtck} coincides with the \cite[Theorem 11]{GrantCweights}.
\end{remark}

We now identify instances when $m$ does not divide $\binom{2n}{n}$.
Consequently, we will have cases where $m$ does not divide $\binom{2n}{n}^{\lfloor k/2\rfloor}$.
We begin by describing some notations.
Let $n=n_0+n_1p+\cdots+ n_kp^k$ be the $p$-adic representation of $n$, where $n_i\in \{0,\dots,p-1\}$ and \[r_i(n)=\begin{cases}
    0& 2n_i< p\\
    r& 2n_i\geq p,
\end{cases}\]   
where
\begin{equation}\label{eq: def r}
r=\text{max}\{j \mid n_{i+1}=n_{i+2}=\cdots=n_{i+j-1}=(p-1)/2\}.    
\end{equation}
With these above notations, Farber and Grant in \cite[Equation 14]{GrantCweights} associated the non-negative integer to the $p$-adic representation of an integer $n$ as follows 
\[\alpha_p(n)=\sum_{i=0}^{k}r_i.\]

We use this information to obtain the exact value of higher topological complexity of lens spaces in certain cases.
\begin{theorem}\label{thm: exact tck 1}
Let $m$ and $n$ be integers such that $p^{\alpha_p(n)\cdot \lfloor k/2\rfloor+1}$ divides $m$ for some prime $p$. 
Then 
\begin{equation}\label{eq: tck exact value1}
 \TC_k(L^{2n+1}_m)=k\cdot (2n+1)=k\cdot\mathrm{dim}(L^{2n+1}_m).
\end{equation}
\end{theorem}
\begin{proof}
Let $p$ be a prime integer. Then it follows from \cite[Lemma 19]{GrantCweights} that the highest power of $p$ dividing $\binom{2n}{n}$ is $\alpha_p(n)$. 
Therefore, the highest power of $p$ dividing $\binom{2n}{n}^{\lfloor k/2\rfloor}$ is $\alpha_p(n)\cdot \lfloor k/2\rfloor$.
Since $p^{\alpha_p(n)\cdot \lfloor k/2\rfloor+1}$ divides $m$, we get that $m$ doesn't divide $\binom{2n}{n}^{\lfloor k/2\rfloor}$. 
Therefore, we can have $l=l'=n$ in \eqref{eq: lbtck}.
Suppose $k$ is even. Then from \eqref{eq: lbtck} we get the following \[k(2n)+k=k(2n+1)\leq \TC_k(L^{2n+1}_m).\] 
Now the equality in \eqref{eq: tck exact value1} follows from the \Cref{thm: lens tck ub}. 

We now assume that $k$ is odd. Then it follows from \eqref{eq: lbtck} that 
\[(k-1)2n+k +2n= k(2n+1)\leq \TC_k(L^{2n+1}_m).\]
The equality in \eqref{eq: tck exact value1} follows from \Cref{thm: lens tck ub}. 
\end{proof}

\begin{theorem}\label{thm: tck exact 2}
Let $n=n_0+n_1p+\cdots+ n_kp^k$ be the $p$-adic representation of $n$, where $n_i\in \{0,\dots,p-1\}$ and $p\geq 3$ such that $n_i\leq (p-1)/2$. 
Then 
\[\TC_k(L^{2n+1}_p)=k\cdot (2n+1)= k\cdot\mathrm{dim}(L^{2n+1}_m). \]
\end{theorem}
\begin{proof}
Let $n=n_0+n_1p+\cdots+ n_kp^k$ 
be the $p$-adic representation of $n$ such that $n_i\leq (p-1)/2$. Then one can observe that $r_i(n)=0$ and consequently $\alpha_p(n)=0$. This shows that the hypothesis of \Cref{thm: exact tck 1} is satisfied. Then theorem follows.
\end{proof}

\begin{theorem}\label{thm: tck exact 3}
Let $m,n\in \Z$ such that $m=2^r$ and $\alpha_2(n)\cdot\lfloor k/2\rfloor \leq r-1$. Then 
\[\TC_k(L^{2n+1}_m)=k\cdot (2n+1) = k\cdot\mathrm{dim}(L^{2n+1}_m) .\]
\end{theorem}
\begin{proof}
It is a consequence of \cite[Lemma 19]{GrantCweights} that the highest power of $2$ can divide $\binom{2n}{n}^{\lfloor k/2 \rfloor}$ is $\alpha_2(n)\cdot \lfloor k/2 \rfloor$. Since $\alpha_2(n)\cdot\lfloor k/2\rfloor\leq r-1$, it follows that $m=2^r$ cannot divide $\binom{2n}{n}^{\lfloor k/2\rfloor}$. Therefore, theorem follows from \Cref{thm: exact tck 1}.
\end{proof}

Now one can easily show that the following is a consequence of previous results.
\begin{corollary}\label{thm: tck exact 4}
Let $m\in \Z$ such that $m\geq 3$. Then $\TC_k(L^3_m)=3k$. 
\end{corollary}

\noindent\textbf{Concluding remarks:} For a fibration $F\to E\to B$,  we have obtained the upper bounds on $\TC_k(E)$ in \Cref{thm: tck ub total space sub tc}  and then in \Cref{cor: ub tck}, we compared the bound in \eqref{eq: ub total space} with the usual dimensional upper bound.
It is natural to ask, under what conditions the bounds obtained in \Cref{thm: tck ub total space sub tc} improves the dimension-connectivity upper bound \eqref{eq: dim conn ub}?

In \Cref{sec: gact}, we used group actions to improve the usual dimensional upper bound on the higher topological complexity. One can ask, how can the group actions be used to improve the dimension-connectivity upper bound?

\vspace{.5cm}

\noindent\textbf{Acknowledgment:}
The author extends gratitude to both reviewers for their valuable suggestions and comments, which greatly contributed to enhancing the paper's presentation and exposition. Special thanks are also due to the reviewer for suggesting questions highlighted in the concluding remarks. 
Author also thank Prof. Jes\'{u}s Gonz\'{a}lez for his comments and insightful discussions related to this work and appreciation is expressed to Prof. Rekha Santhanam and Soumyadip Thandar for their many useful discussions.

\bibliographystyle{plain} 
\bibliography{references}

\end{document}